\documentclass[11pt]{amsart}
\usepackage{amsmath,amsfonts,amsthm,amssymb,nameref}
\usepackage[top=4cm, bottom=3cm, left=3.3cm, right=3.3cm]{geometry} 

\usepackage{tikz} 
\usetikzlibrary{calc,automata} 

\usepackage{color} 

\usepackage{enumerate}

\usepackage[alphabetic]{amsrefs}

\usepackage[pdftex]{hyperref}
\hypersetup{colorlinks,citecolor=black,filecolor=black,linkcolor=black,urlcolor=black}
\usepackage[all]{hypcap}
\usepackage{microtype} 

\newcommand{\N}	{\mathbb N}
\newcommand{\Z}	{\mathbb Z}
\newcommand{\R}	{\mathbb R}

\newcommand{\diam}	{\operatorname{diam}} 

\newcommand{\AG}	{A_\Gamma}
\newcommand{\Sph}	{\mathbb S^2}

\newcommand{\sier}{Sierpi\'nski }

\newtheorem{theorem}{Theorem}[section]
\newtheorem{proposition}[theorem]{Proposition}
\newtheorem{lemma} [theorem]{Lemma} 
\newtheorem{cor} [theorem]{Corollary}
\newtheorem*{theorem*} {Theorem} 
\newtheorem*{prop*}{Proposition}
\newtheorem*{lem*} {Lemma} 
\newtheorem*{cor*} {Corollary}
\theoremstyle{definition}

\newtheorem{definition}[theorem]{Definition}
\newtheorem{example}[theorem]{Example}
\newtheorem{remark}[theorem]{Remark}
\newtheorem*{defi*}{Definition}
\newtheorem*{example*}{Example}
\newtheorem*{remark*}{Remark}
\newtheorem*{problem*}{Problem}
\newtheorem*{convention*}{Convention}

\newtheoremstyle{citing}
  {3pt}
  {3pt}
  {\itshape}
  {}
  {\bfseries}
  {}
  {.5em}
  {\thmnote{#3}}

\theoremstyle{citing}
\newtheorem*{varthm}{}

\author{Ruth Charney}
\author{Matthew Cordes}
\author{Alessandro Sisto}

\title{Complete topological descriptions of certain Morse boundaries}

\begin{document}

\thanks{Charney was partially supported by NSF grant DMS-1607616.}  
\thanks{Cordes was partially supported by the ETH Zurich Postdoctoral Fellowship Program, cofunded by a Marie Curie Actions for People COFUND Program.}
\thanks{Sisto was partially supported by the Swiss National Science Foundation (grant \#182186), as well as partially supported by the grant 346300 for IMPAN from the Simons Foundation and the matching 2015-2019 Polish MNiSW fund.}

\maketitle

\begin{abstract}
We study direct limits of embedded Cantor sets and embedded \sier curves.  We show that under appropriate conditions on the embeddings, all limits of Cantor spaces give rise to homeomorphic spaces, called $\omega$-Cantor spaces,  and similarly, all limits of \sier curves give homeomorphic spaces, called to $\omega$-\sier curves.   We then show that the former occur naturally as Morse boundaries of right-angled Artin groups and fundamental groups of non-geometric graph manifolds, while the latter occur as Morse boundaries of fundamental groups of  finite-volume, cusped hyperbolic 3-manifolds. 
\end{abstract}

\section{Introduction}

Many geodesic metric spaces which are not Gromov hyperbolic nonetheless display some hyperbolic-like behavior.  The Morse boundary was introduced in \cite{charney-sultan} and  \cite{CordesMorseBoundaries} with the goal of identifying and encoding this behavior in a useful way. It is defined for any proper geodesic metric space (though in some cases it may be empty), and a key property states that is invariant under quasi-isometries.  In particular, the Morse boundary of a finitely generated group $G$ is well-defined and provides a quasi-isometry invariant for the group.

Set theoretically, the Morse boundary of $X$ consists of equivalence classes of geodesic rays satisfying the Morse property, namely, the property that any $(\lambda,\epsilon)$-quasi-geodesics with endpoints on the ray stays bounded distance from the ray, where the bound, $N = N(\lambda, \epsilon)$, depends only on $\lambda$ and $\epsilon$.  If one restricts the Morse gauge $N$, then the corresponding boundary points can be given a natural topology.  This space is denoted $\partial_M^N X$. The Morse boundary $\partial_M X$ is defined as the direct limit of these spaces.   (See Section \ref{sec:background} below for formal definitions.)  

It is shown in \cite{cordes-hume} that the intermediate strata $\partial_M^N X$ are, in essence, boundaries of hyperbolic spaces and thus have many nice properties.  Using the direct limit topology on $\partial_M X$ has the advantage that many questions can be reduced to understanding these intermediate spaces.  Indeed, the Morse boundary has been shown to have many properties analogous to boundaries of hyperbolic spaces (see eg., \cite{cordes-hume}, \cite{quasi-mobius}, \cite{Murray:dynamics}, \cite{Liu:dynamics}, \cite{Zalloum}).  On the other hand, the limit space itself can be quite complicated. In particular, if $X$ itself is not a hyperbolic space, then the topology on $\partial_M X$ is neither compact nor metrizable. However, in this paper we provide the first complete topological descriptions of non-compact Morse boundaries, showing that, at least sometimes, Morse boundaries are more ``accessible'' as topological spaces than previously thought.

More specifically, we study the Morse boundaries of right-angled Artin groups, fundamental groups of non-geometric graph manifolds, and fundamental groups of finite-volume cusped hyperbolic 3-manifolds.  Understanding these boundaries involves not only characterizing the strata $\partial_M^N X$, but also understanding how they sit inside each other.   As we will discuss below, the process is more subtle than it at first appears.

\subsection*{Direct limits of Cantor spaces and right-angled Artin groups}
For the case of a right-angled Artin group we prove:

\begin{theorem}\label{thm:Artin_groups}
 Let $A_\Gamma$ be a right-angled Artin group, for $\Gamma$ a finite graph. Then $\partial_M A_\Gamma$ satisfies exactly one of the following:
 \begin{enumerate}
  \item $\partial_M A_\Gamma$ is empty,
  \item $\partial_M A_\Gamma$ consists of two points,
  \item $\partial_M A_\Gamma$ is a Cantor space, or
  \item $\partial_M A_\Gamma$ is an $\omega$-Cantor space.
 \end{enumerate}
\end{theorem}

In addition, we give precise conditions on the defining graph $\Gamma$ for when each of these cases occur.  Here, an $\omega$-Cantor space is defined as a direct limit of a sequence of Cantor spaces $X_1 \subset X_2 \subset X_3  \dots$ such that $X_i$ has empty interior in $X_{i+1}$ for all $i$. We call these limit spaces $\omega$-spaces because all $\sigma$-compact Morse boundaries are examples of $k_\omega$-spaces (for more information see \cite{k-omega}).  Crucially, in Theorem \ref{thm:limit_Cantor}, we show that any two $\omega$-Cantor spaces are homeomorphic, so that (4) describes a well-defined homeomorphism type.  

We also show that nice graph of group decompositions  have $\omega$-Cantor space boundaries:

\begin{theorem} \label{thm:Bass--Serre_is_Cantor}
Let $G$ be a group with a graph of groups decomposition $\mathcal{G}$ such that the vertex groups are undistorted in $G$ and have empty Morse boundary, and $G$ acts acylindrically on the Bass--Serre tree associated to the decomposition. Then $\partial_M G$ is an $\omega$-Cantor space.
\end{theorem}

An immediate corollary of this theorem is that non-geometric graph manifolds have $\omega$-Cantor space boundaries:

\begin{cor}\label{thm:graph_mflds}
 Let $M$ be a non-geometric graph manifold and let $G=\pi_1(M)$. Then $\partial_M G$ is an $\omega$-Cantor space.
\end{cor}

These results in fact follow from a more general theorem:

\begin{theorem}\label{thm:boundary_is_Cantor}
 Let $G$ be a finitely generated group. Suppose that $\partial_M G$ is totally disconnected, $\sigma$-compact, and contains a Cantor subspace.  Then $\partial_M G$ is either a Cantor space or an $\omega$-Cantor space.   It is a Cantor space if and only if  $G$ is hyperbolic, in which case $G$ is virtually free.
\end{theorem}



We note that all hypotheses on $\partial_M G$ are necessary, but for the last two there is no known counterexample, for more discussion see below.

At first sight, one might think that Theorem \ref{thm:boundary_is_Cantor} should be straightforward to prove using the stratification of $\partial_M G$ given by Morse gauges. By hypothesis, each stratum is a totally disconnected compact space, which should make it a Cantor space, and hence would make $\partial_M G$ a direct limit of Cantor spaces, as required. However, the fact that strata are compact and totally disconnected does not mean that they are Cantor spaces; they could be any closed subspace of a Cantor space, of which there are uncountably many homeomorphism types. The solution is to perturb the stratification by Morse gauges to ensure that each stratum is a Cantor space.  To do this we need to, roughly speaking, ``add'' a sequence of Cantor spaces converging to each isolated point of the stratum. This sequence of Cantor spaces is obtained as translates of the Cantor space given in the hypothesis of the theorem.  

The hypothesis that $\partial_M G$ is $\sigma$-compact is necessary since the $\omega$-Cantor space is $\sigma$-compact by construction. On the other hand, to the best of our knowledge there is no known example of a group $G$ so that $\partial_M G$ is not $\sigma$-compact. However, we believe that such examples can be found among small-cancellation groups, and it is even possible that \emph{every} infinitely presented $C'(1/6)$-small-cancellation group has non-$\sigma$-compact Morse boundary.

The hypothesis that $\partial_M G$ contains a Cantor subspace is also a necessary condition. However, there is no known example of a group with non-compact Morse boundary that does not contain a Cantor subspace. In fact, in most of the motivating examples, one has a stable free subgroup (e.g., in acylindrically hyperbolic groups). On the other hand, there are examples by E.~Fink of torsion groups with Morse rays \cite{Fink}, but even those probably contain a Cantor subspace. In fact, Fink finds a tree in the Cayley graph where all rays are Morse. We believe that this tree is most likely stable, thereby exhibiting a Cantor subspace in the Morse boundary.   This naturally raises a question: Does there exist a group whose Morse boundary is non-compact and does not contain a Cantor subspace?

\subsection*{Direct limits of \sier curves and cusped hyperbolic 3-manifolds}
In the case of fundamental groups of cusped hyperbolic 3-manifolds, we show that their Morse boundaries are direct limits of a sequence $S_1 \subset S_2 \subset S_3 \dots $ of \sier curves. Moreover, these curves are nicely embedded, namely the peripheral curves of $S_i$ are disjoint from those of $S_{i+1}$.  We call such a limit an $\omega$-\sier curve and prove that any two such are homeomorphic.  

\begin{theorem}\label{thm:Sierpinski_boundary}
 Let $M$ be a finite-volume hyperbolic $3$-manifold with at least one cusp, and let $G=\pi_1(M)$. Then $\partial_M G$ is an $\omega$-\sier curve.
\end{theorem}

The difficulty here is similar to the problems that arise in the proof of Theorem \ref{thm:boundary_is_Cantor}.  Roughly, each stratum is obtained from $\Sph=\partial \mathbb H^3$ by removing ``shadows'' of horoballs, so one might at first expect that strata are \sier curves. However, the shadows are not disjoint, so the situation is more complicated, and once again we have to perturb the stratification. In this case, more technology is needed to do so, namely technology from \cite{Mac-08-quasi-arc, MS:quasi_planes}, which allows us to carefully detour all the boundary circles of the shadows in order to make them disjoint. This is a rather delicate procedure, and in fact the technology to make something like this work in higher dimensions seems to not be available. It would be interesting to know whether Morse boundaries of fundamental groups of non-compact, finite-volume hyperbolic $n$-manifolds are all homeomorphic to a certain limit of the $(n-1)$-dimensional analogues of \sier curves.

One corollary of the fact the Morse boundary of a finite-volume hyperbolic $3$-manifold with at least one cusp is an  $\omega$-\sier curve, is that the boundary is not totally disconnected. Tran conjectured that the Morse boundary of a right-angled Coxeter group is totally disconnected if and only if every induced cycle of length greater than four in the defining graph contains a pair of non-adjacent vertices that are contained in an induced 4-cycle \cite{Tran}.  As Tran pointed out to us, the one-skeleton of a 3-cube satisfies this condition and its associated right-angled Coxeter group  is virtually a finite volume hyperbolic three manifold with cusps \cite{Haulmark-Nguyen-Tran}.  It follows from our theorem that this is a counterexample to Tran's conjecture.  We note that this is not the first counterexample to this conjecture, see \cites{Graeber-Karrer-Lazarovich-Stark, Karrer}.

\subsection*{Outline}  Section 2 contains definitions and background on Morse boundaries.  In Sections 3 and 4 we investigate limits of Cantor sets and apply this to characterize Morse boundaries of right-angled Artin groups.  In Sections 5 and 6, we investigate limits of \sier curves and prove Theorem \ref{thm:Sierpinski_boundary}.

\section{Background}\label{sec:background}

\subsection{Morse boundary}

\emph{We will assume throughout that $X$ is a proper geodesic metric space.}  We begin with a definition of the Morse boundary of $X$ and some properties that will be needed for the arguments below.  We refer the reader to \cite{CordesMorseBoundaries} for more details.

\begin{definition}  A geodesic $\alpha$ in $X$ is \emph{Morse} if there exists a function $N: \R^+ \times \R^+ \to \R^+$ such that any $(\lambda, \epsilon)$-quasi-geodesic with endpoints on $\alpha$,  lies in the $N(\lambda, \epsilon)$-neighborhood of $\alpha$.  The function $N$ is called a \emph{Morse gauge} for $\alpha$.
 \end{definition}
 
 \begin{lemma}[Lemma 3.1 \cite{Liu:dynamics}] \label{lem:morse subgeo}
 For each Morse gauge $N$, there exists an $N'$ depending only on $N$, such that if $\alpha$ is an $N$-Morse geodesic in $X$ and $\beta$ is a subgeodesic of $\alpha$, then $\beta$ is $N'$-Morse.
 \end{lemma}
 
\begin{remark} 
Taking Lemma \ref{lem:morse subgeo} into consideration, we call a geodesic $N$-Morse if the geodesic and all subgeodesics have Morse gauge $N$.
\end{remark}

It is shown in \cite{CordesMorseBoundaries} that if two sides of a triangle in $X$ are $N$-Morse, then the third side is $N_1$-Morse where $N_1$ depends only on $N$.  This also holds for ideal triangles by \cite{quasi-mobius}, Lemma 2.3.  We record this property here as it will be used repeatedly below.

\begin{lemma}[\cite{quasi-mobius}] \label{lem:triangles} Let $X$ be a proper geodesic metric space. Let  $\Delta(x, y, z)$ be a geodesic triangle with vertices $x, y, z \in X \cup \partial_M X$ and suppose that two sides of $\Delta$ are $N$-Morse.  Then the third side is $N'$-Morse where $N'$ depends only on $N.$
\end{lemma}

 For two Morse rays $\alpha, \beta$ in $X$, say $\alpha \sim \beta$  if they have bounded Hausdorff distance.  It is shown in \cite{CordesMorseBoundaries} that this bound depends only on the Morse gauge $N$, that is, there exists $C_N$ such that two $N$-Morse rays $\alpha, \beta$ are equivalent if and only if  $d(\alpha(t), \gamma(t))< C_N$ for all $t$. The Morse boundary of $X$ consists of the set of equivalence classes of Morse rays.  To topologize this set, first choose a basepoint $e \in X$ and let $N$ be a Morse gauge.  Set 
\begin{equation*} \partial^N_M X_e= \{[\alpha] \mid \exists \beta \in [\alpha] \text{ that is an $N$--Morse geodesic ray with } \beta(0)=e\} \end{equation*} 
with the compact-open topology. These spaces are compact.  This topology is equivalent to one defined by a system of neighborhoods, $\{V_n(\alpha) \mid n \in \N \}$, defined as follows: $V_n( \alpha)$ is the set of $[\gamma] \in \partial ^N_M X_e$ such that   $d(\alpha(t), \gamma(t))< C_N$ for all $t<n$. 

Let $\mathcal M$ be the set of all Morse gauges. Put a partial ordering on $\mathcal M$ so that  for two Morse gauges $N, N' \in \mathcal M$, we say $N \leq N'$ if and only if $N(\lambda,\epsilon) \leq N'(\lambda,\epsilon)$ for all $\lambda,\epsilon \in \N$.  Define the \emph{Morse boundary} of $X$ to be
 \begin{equation*} \partial_M X =\varinjlim \partial ^N_M X_e
 \end{equation*} 
with the induced direct limit topology, i.e., a set $U$ is open in $\partial_M X$ if and only if $U \cap \partial ^N_M X_e$ is open for all $N$.   A change in basepoint results in a homeomorphic boundary, justifying the omission of the basepoint from the notation. More generally, we will usually assume the basepoint is fixed and omit it from the notation $\partial ^N_M X_e$ as well.  The reader is warned, however, that unlike $\partial_M X$, these subspaces do depend on a choice of basepoint.

An alternate construction of the Morse boundary is given by the second author and Hume in \cite{cordes-hume}.   Define $X^{(N)}_e$ to be the set of all $y\in X$ such that there exists a $N$--Morse geodesic $[e,y]$ in $X$. Then $X^{(N)}_e$ satisfies the Gromov 4-point definition of hyperbolicity, hence, we may consider its Gromov boundary, $\partial X^{(N)}_e$, and the associated visual metric $d_{(N)}$. The collection of boundaries $\left( \partial X^{(N)}_e, d_{(N)} \right)$ is called the metric Morse boundary of $X$.  
 It is shown in \cite{cordes-hume} that there is are natural homeomorphisms between $\partial X^{(N)}_e$ and $\partial ^N_M X_e$ that induce a homeomorphism on the direct limits. Thus the Morse boundary, $\partial_M X$,  can also be thought of as the direct limit of the metric spaces $\partial X^{(N)}_e$.

We now establish some useful properties of these spaces.
 
 \begin{lemma} \label{lem:top embedding}
For any $N, N'$ with $N<N'$, the inclusion $i \colon \partial_M^N X \hookrightarrow \partial_M^{N'}X$ is a topological embedding.
 \end{lemma}
 
 \begin{proof}
 This map is continuous by Corollary 3.2 of \cite{CordesMorseBoundaries}. What is left to show is that the map is closed. Let $K$ be a closed set in $\partial_M^{N'}X$. Let $\gamma_i$ be a sequence in $i(K)$ converging to a point $\gamma$. Since the $\gamma_i$ are represented by $N$-Morse geodesics and $N$-Morse geodesics converge to $N$-Morse geodesics \cite[Lemma 2.10]{CordesMorseBoundaries} it converges to an $N$-Morse geodesic $\alpha$ in $\partial_M^{N'}X$. But since $i$ is simply the inclusion and $K$ is closed in $\partial_M^N X$, then $\alpha = \gamma$ and thus $i(K)$ is closed.
 \end{proof}

 The following is a general fact about $\sigma$-compact Morse boundaries. It says that one can choose a countable exhaustion of the Morse boundary by strata.
 
 \begin{lemma}\label{lem:increasing_gauge}
 If $\partial_M X$ is $\sigma$-compact, then there exists a sequence  of Morse gauges $N_1,N_2,\dots$ with $N_i\leq N_{i+1}$ so that $\partial_M X=\bigcup \partial X^{(N_i)}_e = \bigcup \partial^{N_i}_M X$.
 \end{lemma}
 
 \begin{proof}
 By definition, since $\partial_M X$ is  $\sigma$-compact, then $\partial_M X = \bigcup_{i \in \N} K_i$ where $K_i$ are compact and $K_i \subseteq K_{i+1}$. By Lemma 4.1 of \cite{convex-cocompact} we know that for any compact subset $K \subset \partial_M X$ there exists an $N$ so that $K \subset \partial_M X^{(N)}_e$. Thus for each $K_i$ we have a Morse gauge $N_i$ and the ascending condition on the $K_i$ guarantees that $N_i\leq N_{i+1}$ and since the $K_i$ exhaust $\partial_M X$ then it follows that $\partial_M X=\bigcup \partial X_e^{(N_i)}$.
 \end{proof}

\section{Limits of Cantor spaces}\label{sec:limit_cantor}

In this section we describe particular direct limits of Cantor sets, and more precisely of sequences of Cantor spaces each having empty interior in the next one.

\begin{definition}
 Let $C\subsetneq D$ be Cantor spaces. We say that $C$ is \emph{entwined} in $D$ if $C$ has empty interior in $D$. An $\omega$-Cantor space is a topological space $\varinjlim_{i\in \mathbb N} X_i$, where each $X_i$ is a Cantor space and each $X_i$ is entwined in $X_{i+1}$.
\end{definition}

\begin{example}
Let $C$ be the usual middle-third Cantor space. For each middle third interval $I$,  glue in a copy $C'$ of $C$ by identifying the boundary points of $C'$ with the boundary points of $I$.  Call the resulting space $D$. Then $D$ is again a Cantor space and $C$ is entwined in $D$.  

A similar example can be described in terms of boundaries.  Let $T$ be a (rooted) trivalent tree so that $\partial T$ can be identified with $C$.  Now glue a separate copy of $T$ to each vertex $v$ of $T$ by identifying the root to $v$.  The result is a larger tree $T'$ with $T \subset T'$ such that $\partial T$ is entwined in $\partial T'$.  
\end{example}

The main result of this section is that there is only one $\omega$-Cantor space up to homemorphism:

\begin{theorem}\label{thm:limit_Cantor}
 Any two $\omega$-Cantor spaces are homeomorphic.
\end{theorem}

The rest of this section is devoted to the proof of this theorem. After the preliminary observation in the remark below, we show that, roughly speaking, when $C\subseteq D$ are Cantor spaces, with $C$ entwined in $D$, and $C$ is subdivided in two clopen sets $C_0,C_1$, we can take small clopen neighborhoods of $C_0,C_1$ in $D$.

\begin{remark}\label{rem:clopen_Cantor}
 A clopen set in a Cantor space is either empty or a Cantor space.
\end{remark}

\begin{lemma}\label{lem:extend_clopen}
 Let $C\subseteq D$ be Cantor spaces, with $C$ entwined in $D$. Then if $C=C_0\sqcup C_1$ with $C_i$ clopen and non-empty, then there exist clopen sets $D_0,D_1\subseteq D$ so that:
  \begin{enumerate}
  \item $C_i=D_i\cap C$,
  \item $D_0\cap D_1=\emptyset$,
  \item $D_0\cup D_1\subsetneq D$.
  \item If $d$ is any metric on $D$ compatible with its topology, then for any $\epsilon >0$, we can choose $D_i\subseteq N_{\epsilon}(C_i)$.
  \item $C_i$ and $D_i$ are Cantor spaces, and $C_i$ is entwined in $D_i$.
 \end{enumerate}
\end{lemma}

\begin{proof}
 Fix a metric $d$ on $D$ compatible with its topology and $\epsilon>0$. We can choose $\epsilon'>0$ so that $D'_i=N_{\epsilon'}(C_i)$ satisfy the first 4 properties listed above (in fact, up to now we only need $C$ to be properly contained in $D$). However, the $D'_i$ are open, but they might be not closed; we now shrink them to make them clopen. Since every point in a Cantor space has a neighborhood basis of clopen sets, for every point $x$ in $D'=D-(D'_0\cup D'_1)$ there exists a clopen set $U_x$ disjoint from $C_0\cup C_1$ (we also use that $C_0\cup C_1$ is closed). Since $D'$ is compact, it is covered by finitely many $U_{x_i}$, and the union $U$ of such $U_{x_i}$ is clopen. Then, $D_i=D'_i-U$ is open (because it is obtained removing a closed set from an open set) and closed (since the complement is $D'_{1-i}\cup U$, which is open), and it is readily seen that the first 4 properties are satisfied.
 
 The fact that $C_i$ and $D_i$ are Cantor spaces follows now from Remark \ref{rem:clopen_Cantor}. We now argue that $C_i$ is entwined in $D_i$. Notice that if $U$ is open in $D_i$ then it is also open in $D$, since $D_i$ is open. In particular, the interior of $C_i$ as a subspace of $D_i$ coincides with the interior of $C_i$ as a subspace of $D$. Since $C\supseteq C_i$ has empty interior in $D_i$, we have that $C_i$ has empty interior in $D$, whence in $D_i$, as required.
\end{proof}

Next, we show that we can extend partial homeomorphisms of Cantor spaces defined over entwined Cantor subspaces.

\begin{lemma}\label{lem:extend_homeo}
 Let $C\subseteq D$ and $C'\subseteq D'$ be Cantor spaces, with $C$ entwined in $D$ and $C'$ entwined in $D'$. Let $\phi:C\to C'$ be a homeomorphism. Then there exists a homeomorphism $\overline\phi:D\to D'$ that extends $\phi$.
\end{lemma}

\begin{proof}
 Let $\Omega$ be the set of finite words in the alphabet $\{0,1\}$, and denote by $e$ the empty word. We denote the word obtained by appending $0$ (resp. 1) at the end of the word $\omega\in\Omega$ simply by $\omega0$ (resp. $\omega1$). Fix metrics on $D$ and $D'$ compatible with their respective topologies. We can choose clopen sets $C_{\omega}\subseteq C$ with the following properties:
 \begin{enumerate}
  \item $C_e=C$,
  \item $C_{\omega0},C_{\omega1}\subseteq C_{\omega}$,
  \item $C_{\omega}=C_{\omega0}\sqcup C_{\omega1}$,
  \item for any $\epsilon>0$, there are only finitely many $C_{\omega}$ whose diameter is larger than $\epsilon$.
 \end{enumerate}
 
 Set $C'_{\omega}=\phi(C_{\omega})$, and notice that analogous properties hold for the $C'_{\omega}$ as well. Also, set $D_e=D$, and, inductively on the length of $\omega$, use Lemma \ref{lem:extend_clopen} to construct clopen sets $D_{\omega}\subseteq D$ satisfying:
 \begin{enumerate}
  \item $D_{\omega}\cap C=C_{\omega}$,
  \item $D_{\omega0},D_{\omega1}\subseteq D_\omega$,
  \item $D_{\omega0}\cap D_{\omega1}=\emptyset$,
  \item $K_\omega=D_{\omega}-(D_{\omega0}\cup D_{\omega1})$ is non-empty,
  \item for any $\epsilon>0$, there are only finitely many $D_{\omega}$ whose diameter is larger than $\epsilon$.
 \end{enumerate}
 Also, construct $D'_{\omega}\subseteq D'$ (and $K'_{\omega}$) with the analogous properties. Notice that $D-C=\bigsqcup_{\omega\in\Omega} K_{\omega}$ and that each $K_{\omega}$ (which is clopen) is a Cantor space by Remark \ref{rem:clopen_Cantor}. Similar observations hold for $D'-C'$ and $K'_{\omega}$.
 
 We are now ready to define $\overline \phi$. Choose for each $\omega$ any homeomorphism $\phi_{\omega}:K_{\omega}\to K'_{\omega}$, and set $\overline \phi(x)=x$ if $x\in C$, and $\overline \phi(x)=\phi_{\omega}(x)$ if $x\in K_\omega$. Then $\overline\phi$ is bijective. Hence, since $D,D'$ are compact and Hausdorff, to show that $\overline \phi$ is a homeomorphism we are left to show that it is continuous. Notice that by construction, we have $\overline\phi(D_{\omega})=D'_\omega$ for each $\omega$. It is readily checked that $\{D'_\omega\}$ is a basis for the topology of $D'$. In fact, for any $x\in D'$ and $\epsilon>0$ there exists some $D'_\omega$ of diameter at most $\epsilon$ that contains $x$, since there are infinitely many $D'_\omega$ containing $x$ and only finitely many of them have diameter larger than $\epsilon$. Hence, the preimage under $\overline\phi$ of any member of a neighborhood basis for $D'$ is open, and hence $\overline\phi$ is continuous.
\end{proof}

Finally, we are ready to prove Theorem \ref{thm:limit_Cantor}, by iterative extensions.

\begin{proof}[Proof of Theorem \ref{thm:limit_Cantor}]
 Let $X=\varinjlim_{i\in \mathbb N} X_i$ and $X'=\varinjlim_{i\in \mathbb N} X'_i$ be $\omega$-Cantor spaces. Start with any homeomorphism $\phi_0:X_0\to X'_0$ and  inductively define homeomorphisms $\phi_{n+1}:X_{n+1}\to X'_{n+1}$ extending $\phi_n$; those exist by Lemma \ref{lem:extend_homeo}. Then the $\phi_n$ give a well-defined bijection $X\to X'$, which is a homeomorphism by definition of the direct limit topology.
\end{proof}

\section{$\omega$-Cantor boundaries}

We now study totally disconnected Morse boundaries of finitely generated groups. In this case, we write $\partial_M G$ for the Morse boundary of some (hence any) Cayley group of $G$.  

Following \cite{cordes-hume}, we say that a quasi-convex subspace $Y$ of a geodesic metric space $X$ is \emph{$N$-stable} if every pair of points in $Y$ can be connected by a geodesic which is $N$-Morse in $X$. We say that a subgroup is stable if it is stable as a subspace.

The main theorem in this section is the following. 
\begin{varthm}[Theorem \ref{thm:boundary_is_Cantor}]
 Let $G$ be a finitely generated group. Suppose that $\partial_M G$ is  totally disconnected, $\sigma$-compact, and contains a Cantor subspace.  Then it is either a Cantor space or an $\omega$-Cantor space.  It is a Cantor space if and only if  $G$ is hyperbolic, in which case $G$ is virtually free.
\end{varthm}

We will need the following lemma to construct sequences limiting to a specific boundary point, to show that such boundary point is not isolated.

\begin{lemma}\label{lem:lim}
  Let $G$ be a group, let $\gamma,\alpha_1,\alpha_2$ be Morse rays, with $[\alpha_1]\neq [\alpha_2]$, where all rays are based at $e$. Also, let $g_i=\gamma(i)\in G$ be the sequence of group elements traversed by $\gamma$. Then for some $j\in\{1,2\}$ and some subsequence $i(n)$ we have $[g_{i(n)}\alpha_j]\to [\gamma]$.
\end{lemma}

\begin{proof}
By Lemma \ref{lem:triangles},  if two legs of a geodesic triangle are $N$-Morse then the third leg is $N'$-Morse where $N'$ depends only on $N$.  In our case, this says that there exists $N$, depending on the Morse gauges of $\gamma,\alpha_1,\alpha_2$, so that there exist $N$-Morse geodesic rays $\beta^j_i$ based at $e$ with $[\beta^j_i]=[g_i\alpha_j]$. Moreover, there exists $C$, again depending only on the various Morse gauges, so that all triangles $\gamma([0,i])\cup \beta^j_i\cup g_i\alpha_j$ are $C$-slim \cite[Lemma 2.2]{CordesMorseBoundaries}.

\begin{figure}[h]
 \includegraphics[scale=0.7]{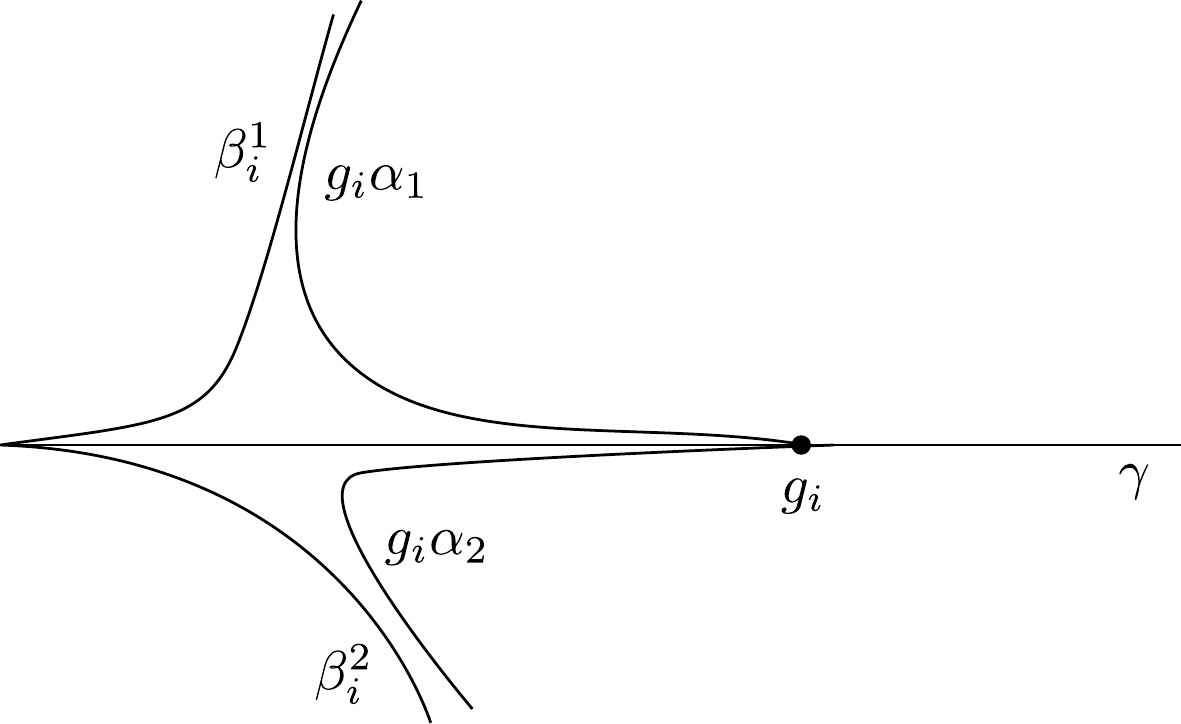}
 \caption{If the $\beta^j_i$ do not converge to $\gamma$, then the $g_i\alpha_j$ both backtrack a lot along $\gamma$. But then, this means that the $\alpha_j$ fellow-travel for arbitrarily long times.}
\end{figure}

 Suppose by contradiction that for $j=1,2$ no subsequence of $\beta^j_i$ converges to $\gamma$. This means that for all $i,j$ the Gromov products $(\gamma,\beta^j_i)_e$ are uniformly bounded by, say, $R$. By slimness, this says that there exists $R'=R'(R,C)$ sush that each $g_i\alpha_j$ intersects  $\mathcal{N}_{R'}(e)$. As a consequence, the Gromov products $(g_i\alpha_1,g_i\alpha_2)_{g_i}$ diverge. However, all said Gromov products are equal to $(\alpha_1,\alpha_2)_{e}$.  This is impossible unless this product is infinite, and hence, $[\alpha_1]=[\alpha_2]$, contradicting one of the hypotheses.
\end{proof}

We now augment stable subsets of $G$ to make their boundary a Cantor space.

\begin{lemma} \label{lem:Cantor_subspace}
 Let $G$ be a group with totally disconnected Morse boundary and a Cantor space $C_0 \subset \partial_M G$. Then for every Morse gauge $N$ there exists an $N'$ and a Cantor space $C\subseteq \partial_M G$ with $\partial_M^N G\subseteq C \subseteq \partial_M^{N'} G$. 
 \end{lemma}
 
 \begin{proof}
Since $C_0$ is compact by \cite[Lemma 4.1]{convex-cocompact} we know that $C_0 \subset \partial_M^{N_0} G$ for some Morse gauge $N_0$. Since $\partial_M G$ is totally disconnected, it follows that for any Morse gauge $N$, so is $\partial_M^N G$. We will consider some $N$ so that  $\partial_M^N G$ is non-empty.  Define
$$A= \partial_M^N G \, \cup (\bigcup_{g\in G^{(N)}_e}  gC_0) \subseteq \partial_M G.$$ 
In fact, by Lemma \ref{lem:triangles}, there exists $N'$ depending on $N,N_0$ only so that  $A\subseteq \partial^{N'}_M G$.

 Let now $C=\overline{A}$, and we claim that $C$ is a Cantor space.  To check this we need show that $C$ is non-empty, perfect, compact, totally disconnected, and metrizable.  It is metrizable, totally disconnected, and compact because it is a closed subspace of $\partial _M^{N'} G$. It is clearly non-empty so what is left to show is that it is perfect, i.e., $C$ has no isolated points. Any point in $C-A$ is clearly not isolated. Also, any point in some $gC_0$ is not isolated, since $gC_0$ is a Cantor space. Finally, any point in $\partial_M^N G$ is a limit of points in $C$ (in fact, even in $A$) by Lemma \ref{lem:lim}.
 \end{proof}
  
 We now show that the subsets $\partial^{N}_M G$ have empty interior in $\partial^{N'}_M G$ whenever $N'\gg N$.
 
 \begin{lemma} \label{lem:entwined}
  If $\partial^{N}_M G\subsetneq \partial_M G$ and $\partial_M G$ is not compact, then there exists $N'$ so that for all $N''\geq N'$, $\partial^{N}_M G$ has empty interior in $\partial^{N''}_M G$.
 \end{lemma}
\begin{proof} Note that it suffices to find $N'$ so that $\partial_M^N G \subseteq \overline{\partial_M^{N'} G-\partial_M^N G}$ because then it automatically follows that for all $N''\geq N'$,  $\partial_M^N G \subseteq \overline{\partial_M^{N''} G -\partial_M^N G}$, which is the same as $\partial_M^{N} G$ having empty interior in $\partial_M^{N''} G$.

Let $\gamma$ be a geodesic in $G$, based at $e$, representing a point $p \in \partial_M^N G$.   Let $\gamma(i)=g_i \in G$.  Then the sequence of points $(g_i)$ converges to $p$.  Passing to a subsequence if necessary, the sequence  $(g_i^{-1})$ also converges to some point $q \in \partial_M G$.  Since $\partial_M G$ is not compact, we can choose a point $z \in \partial_M G$ with $z \neq q$ and $z \notin \partial_M^{N_1} G$, where $N_1$ is the Morse gauge from Lemma \ref{lem:triangles}.
It follows from \cite{Liu:dynamics}, Lemmas 5.3 and 6.9, that for some $N'$,  $g_iz$ converges to $p$ in $\partial_M^{N'} G$.  

It remains to show that $g_i z \notin \partial_M^{N} G$ for all $i$.  To see this, let $\beta_i$ be a geodesic ray based at $e$ representing $g_i z$ and let $\gamma_i$ be the segment of $\gamma$ from $e$ to $g_i$.  Then $\gamma_i, \beta_i, g_i\beta_0$ form a geodesic triangle with vertices $e, g_i, g_iz$.  By Lemma \ref{lem:triangles}, if $\gamma_i$ and $\beta_i$ are both $N$-Morse, then $g_i\beta_0$ is $N_1$-Morse, and hence its translate,  $\beta_0$, is  also $N_1$-Morse.  But $z=[\beta_0]$, so this contradicts our choice of $z$ and we conclude that $\beta_i$ is not $N$-Morse.  Thus we have $g_i z \in {\partial_M^{N'} G-\partial_M^N G}$ for all $i$.
\end{proof}
 
We are ready to prove Theorem \ref{thm:boundary_is_Cantor}.
 
\begin{proof}[Proof of Theorem \ref{thm:boundary_is_Cantor}]
If $\partial_M G$ is compact, then $G$ is hyperbolic by Theorem 4.3 of \cite{convex-cocompact}.  Since $\partial_MG$ is a non-empty, compact, perfect, totally disconnected, metrizable metric space, it follows that its boundary is a Cantor space by \cite{Brouwer}. Since $G$ is a hyperbolic group with Cantor space boundary, $G$ has to be virtually free, see e.g. \cite[Theorem 8.1]{KapovichBenakli} and references therein.

Assume now that $\partial_M G$ is not compact, in which case $G$ is not hyperbolic again by Theorem 4.3 of \cite{convex-cocompact}. Since $\partial_M G$ is $\sigma$-compact we know by Lemma \ref{lem:increasing_gauge} that $\partial_M G = \varinjlim \partial^N_M G$ can be chosen to be a countable limit over gauges $N_1, N_2, \ldots$. By Lemma \ref{lem:Cantor_subspace} there is a Cantor space $C_1$ such that $\partial^{N_1}_M G \subset C_1  \subset \partial^{N_{j(2)}}_M G$ for a sufficiently large $j(2)$. In view of Lemma \ref{lem:entwined}, we may increase $j(2)$ to ensure that $\partial^{N_1}_M G$ has empty interior in $\partial^{N_{j(2)}}_M G$. Proceeding inductively we find a sequence $1=j(1)<j(2)<\cdots$ and Cantor spaces $C_1, C_2, C_3, \ldots$ so that $\partial^{N_{j(i)}}_M G \subset C_i  \subset \partial^{N_{j(i+1)}}_M G$. Since  $\partial^{N_{j(i+1)}}_M G$ has empty interior in $ \partial^{N_{j(i+2)}}_M G$, it follows that $C_i$ is entwined in $C_{i+1}$. Moreover, since $\partial_M G = \varinjlim C_i$ it follows that $\partial_M G$ is an $\omega$-Cantor space.
\end{proof}

 \subsection{Totally disconnected boundary versus totally disconnected levels}
 
 In this subsection we show that, when $\partial_M X$ is $\sigma$-compact, then it is totally disconnected if and only if all its strata are.  This will be needed in our analysis of Morse boundaries of right-angled Artin groups.

\begin{proposition} \label{prop:totally disconnected}
 Let $X$ be a proper metric space, let $e\in X$, and suppose that $\partial_M X$ is $\sigma$-compact. Then $\partial_M X$ is totally disconnected if and only if $\partial^{N}_M X$ is totally disconnected for every Morse gauge $N$.
\end{proposition}

\begin{proof}
 If $\partial_M X$ is totally disconnected then any given stratum, $\partial^{N}_M X$ is totally disconnected as well since the stratum is topologically embedded in $\partial_M X$ by Lemma \ref{lem:top embedding}.
  
 Suppose that $\partial^{N}_M X$ is totally disconnected for every Morse gauge $N$. By Lemma \ref{lem:increasing_gauge}, there exists a sequence of Morse gauges $N_1,N_2,\dots$ with $N_i\leq N_{i+1}$ so that $\partial_M G=\bigcup \partial^{N_i}_M X$.
 
 We will prove that for each pair of distinct points $x,y\in \partial_M X$ there exists a clopen subset $A\subseteq \partial_M X$ so that $x\in A$ and $y\notin A$, which suffices to show that $\partial_M X$ is totally disconnected. We can assume $x,y\in\partial^{N_1}_M X$. In fact, we will construct a sequence of clopen subsets $A_i\subseteq \partial^{N_i}_M X$ so that $x\in A_i$, $y\notin A_i$, and $A_i\cap \partial^{N_j}_M X=A_j$ for each $j<i$; we can then just set $A=\bigcup A_i$. Also, we can replace the last condition simply by $A_i\cap \partial^{N_{i-1}}_M X=A_{i-1}$ (for $i\geq 2$), and the stronger condition with $j<i$ will follow by a simple inductive argument.
 
 Let us start with the remark that, for each $i$, since $\partial^{N_i}_M X$ is a totally disconnected compact metrizable space, we have that each point has a neighborhood basis of clopen subsets. In particular, there exists a clopen set $A_1\subseteq \partial^{N_1}_M X$ with $x\in A_1$ and $y\notin A_1$ (since $\partial^{N_1}_M X$ is Hausdorff).
 
 Suppose that we constructed $A_{i-1}$ with the required properties, for some $i\geq 2$. Set $B_{i-1}=\partial^{N_{i-1}}_M X-A_{i-1}$, and notice that it is also clopen in $\partial^{N_{i-1}}_M X$. Both $A_{i-1}$ and $B_{i-1}$ are closed in $\partial^{N_i}_M X$, since they are compact (because they are closed in the compact space $\partial^{N_{i-1}}_M X$), and $\partial^{N_i}_M X$ is Hausdorff. Since $\partial^{N_i}_M X$ is normal, there exists an open set $U\subseteq \partial^{N_i}_M X$ containing $A_{i-1}$, but not intersecting $B_{i-1}$. Since every point of $\partial^{N_i}_M X$ has a basis of clopen neighborhoods, a straightforward compactness argument gives that $A_{i-1}$ is contained in a finite union of clopen sets of $\partial^{N_i}_M X$, each contained in $U$. The union $A_i$ of said clopen sets has the required properties, and we are done.
\end{proof}

 \subsection{Artin groups}
We now show that $\omega$-Cantor spaces arise naturally as Morse boundaries of right-angled Artin groups.  Let $\AG$ denote the right-angled Artin group with finite defining graph $\Gamma$. 

\begin{varthm}[Theorem \ref{thm:Artin_groups}]
 Let $A_\Gamma$ be a right-angled Artin group, for $\Gamma$ a finite graph. Then $\partial_M A_\Gamma$ satisfies exactly one of the following:
 \begin{enumerate}
  \item $\partial_M A_\Gamma$ is empty,
  \item $\partial_M A_\Gamma$ consists of two points,
  \item $\partial_M A_\Gamma$ is a Cantor space, or
  \item $\partial_M A_\Gamma$ is an $\omega$-Cantor space.
 \end{enumerate}
\end{varthm}

\begin{proof}  If $\AG$ is a direct product ($\Leftrightarrow \Gamma$ is a join), then every geodesic is contained in a flat, so the Morse boundary of $\AG$ is empty.  If $\AG$ is a free group ($\Leftrightarrow \Gamma$ is discrete) then the Morse boundary is equal to the Gromov boundary which is either a Cantor space (if $\AG$ is non-abelian) or just two points (if $\AG=\Z$).  

In all other cases, it is shown in \cite{Koberda-Mangahas-Taylor} that $\AG$ contains a stable non-abelian free group and thus has a Cantor space as a subspace of the Morse boundary. Moreover, $\partial_M \AG$ is $\sigma$-compact since $\AG$ is CAT(0) and thus its Morse boundary is equivalent to the contracting boundary defined in \cite{ charney-sultan}. The boundary of $\partial^{N}_M {\AG}$ is totally disconnected for all Morse gauges $N$ by Theorem 5.1 in \cite{cordes-hume} and thus by Proposition \ref{prop:totally disconnected} the whole Morse boundary is disconnected. 

Applying Theorem \ref{thm:boundary_is_Cantor} gives the desired result.
\end{proof}

\subsection{Graphs of groups} 
\label{subsec:graph}
We now show that a certain class of graphs of groups has $\omega$-Cantor space boundaries. We apply this to show that the Morse boundary of non-geometric graph manifolds are $\omega$-Cantor spaces.

\begin{varthm}[Theorem \ref{thm:Bass--Serre_is_Cantor}]
Let $G$ be a finitely generated group with a graph of groups decomposition $\mathcal{G}$ such that the vertex groups are finitely generated and undistorted in $G$ and have empty Morse boundary. Assume further that $G$ acts acylindrically on the Bass--Serre tree associated to the decomposition. Then $\partial_M G$ is a Cantor space or an $\omega$-Cantor space. It is a Cantor space if and only if  $G$ is hyperbolic, in which case $G$ is virtually free.
\end{varthm}

\begin{proof}
We will apply Theorem \ref{thm:boundary_is_Cantor}, so we need only show that $\partial_M G$ is totally disconnected, $\sigma$-compact, and contains a Cantor subspace. 

Let $\mathcal{G} = (\Gamma, \{G_v\},\{G_e\})$ be the graph of groups decomposition. Choose a maximal tree in $\Gamma$.   Choose a generating set for $G$ consisting of a finite generating set for each vertex group $G_v$ together with a generator $t_e$ for each edge $e$ of $\Gamma$ not contained in the maximal tree.  For edges $e$ in the maximal tree, set $t_e=1$.  Then for any edge $e$ and $h \in G_e$,  we have $h_0 t_e = t_e h_1$ where $h_0, h_1$ are the images of $h$ in the initial and terminal vertex groups of $e$, respectively.  

It will be convenient to work with a certain quasi-isometric model $X$ for $G$.  To construct $X$, 
consider the Bass-Serre tree $T$ whose vertices correspond to cosets of the vertex groups.  For each vertex $v$ in the Bass-Serre tree, take a copy $X_v$ of the Cayley graph of the corresponding vertex group, but with vertices labelled by elements of the coset.  For an edge labelled $e$ in the Bass-Serre tree connecting $v$ to $w$, attach edges between vertices of the form  $gh_0 \in X_v$ and $gh_0t_e = g t_e h_1 \in X_w$ for each $h$ in $G_e$.  This defines the graph $X$.  

Note that $X_v$ is quasi-isometrically embedded in $X$ since we are assuming that vertex groups are undistorted, and the constants are uniform in $v$ since there are only finitely many $G$-orbits of spaces $X_v$.
Moreover, there is a $1$--Lipschitz map $\pi:X\to T$ that maps $X_v$ to $v$, and the preimage of the open star of a vertex $v$ is contained in the $1$--neighborhood of $X_v$. 

We next show that  if $Y \subset X$ is a stable subset, then $Y$ quasi-isometrically embeds in the Bass--Serre tree $T$ via the projection map $\pi \colon X \to T$.  From this it will follow that the map on the boundaries induced by $\pi$, $\partial_M \pi \colon \partial_M X \to \partial T$, is continuous and injective.  

Let $N$ be such that every pair of points of $Y$ can be connected by an $N$-Morse geodesic segment $\gamma$ in $X$. For any such $\gamma$, since $\partial_M X_v$ is empty and $X_v$ is uniformly quasi-isometrically embedded in $X$, we have that $\mathrm{diam}(X_v \cap \gamma)$ is bounded by some constant $D\geq 1$ depending on $N$.  We increase $D$ so that a similar bound also holds when intersecting with the $1$-neighborhood of $X_v$.   We claim that $\pi \circ \gamma$ is a $(D,D)$-quasi-geodesic, which is readily seen to be enough to show that $Y$ quasi-isometrically embeds in $T$.  Since this might be of independent interest, we note that a small variation of the arguments below also proves the following more general lemma:

\begin{lemma}
 Let $G$ be a finitely generated group with a graph of groups decomposition $\mathcal{G}$, and fix a word metric on $G$ and an orbit map $\pi\colon G\to T$, where $T$ is the Bass--Serre tree. Then for each $D_1,D_2\geq 1$ there exists $K\geq 1$ such that the following holds. Let $\gamma$ be a $(D_1,D_1)$--quasi-geodesic in $G$ that intersects each coset of a vertex group in a subset of diameter bounded by $D_2$. Then $\pi\circ \gamma$ is a $(K,K)$--quasi-isometric embedding. 
\end{lemma}

Let $\gamma \colon [a,b] \to X$ be an $N$-Morse geodesic segment  and let $p_v$ be the point where $\gamma$ first enters the $1$-neighborhood of $X_v$ and let $q_v$ be the point where it last leaves said neighborhood. Call this subpath $\gamma_v$. Let $\{v_i\}_{i=1}^n \subset T$ be the collection of vertices on the geodesic from $\pi(\gamma(a))$ to $\pi(\gamma(b))$, so that $d_T(\pi(\gamma(a)),\pi(\gamma(b)))\geq n-1$.
We now argue that the subpaths $\gamma_{v_i}$ must cover $\gamma$. Indeed, fix any $t\in[a,b]$. If $\gamma(t)$ projects to a vertex of the geodesic from $\pi(\gamma(a))$ to $\pi(\gamma(b))$ then this is clear. If $\gamma(t)$ projects to the interior of an edge of said geodesic, then it is still true that $\gamma(t)$ belongs to $\gamma_{v_i}$ for $v_i$ any endpoint of the edge since $\gamma(t)$ belongs to the $1$-neighborhood of $X_{v_i}$ (recall that the $\pi$-preimage of the open star around $v$ is contained in the $1$-neighborhood of $X_v$).

The last remaining case is the one where $\gamma(t)$ does not project to the geodesic from $\pi(\gamma(a))$ to $\pi(\gamma(b))$. Then there is a vertex $v_i$ separating $\pi(\gamma(t))$ from both $\pi(\gamma(a))$ and $\pi(\gamma(b))$. Hence, $\pi\circ\gamma$ has to pass through $v_i$ before and after time $t$, showing that $\gamma(t)$ lies on $\gamma_{v_i}$. In view of this covering, we see 
\[
|b-a| \leq \sum_{i=1}^n d(p_{v_i}, q_{v_i}) \leq D n\leq D(d_T(\pi(\gamma(a)),\pi(\gamma(b)))+1).
\]
Since $\pi$ is 1-Lipschitz, we also have
$d_T(\pi(\gamma(a)),\pi(\gamma(b))) \leq d(\gamma(a),\gamma(b)) = |b-a|$,
so $\pi \circ \gamma$ is $(D,D)$-quasi-geodesic as claimed. 

It follows from this that the map $\partial_M \pi \colon \partial_M X \to \partial T$ is continuous and injective. Since the Gromov boundary of a tree is totally disconnected, we conclude that $\partial_M X$ is also totally disconnected.  To complete the proof of the theorem, it remains to show that $\partial_M X$ is $\sigma$-compact and has a Cantor subspace. 

We first show that $\partial_M X$ has a Cantor subspace. Since $G$ acts acylindrically on a tree, we know by Theorem 6.14 of  \cite{Dahmani-Guirardel-Osin} that we have a hyperbolically embedded free group. By \cite{Sisto} this free subgroup is quasi-convex and thus is stable. We have thus found our stable Cantor subspace.

To see that $\partial_M X$ is $\sigma$-compact we stratify the boundary by looking at collections of geodesics that intersect with the vertex spaces in a bounded amount: let $D \geq 0$ and consider the $x \in \partial_M X$ so that there exists a representative $\gamma$ with the the property that $\mathrm{diam}(X_v \cap \gamma) \leq D$ for all $v$. Call this subset $\partial_D X$. As we noted earlier, any $\partial_M^N X$ is contained in some $\partial_D X$, where $D$ depends on $N$.

We claim that the reverse is also true, that is, $\partial_D X$ is contained in $\partial_M^{N'} X$ for some $N'$. 
This will prove that $\partial_M X$ is $\sigma$-compact since the closures of the $\partial_D X$ provide an exhaustion of $\partial_M X$ by compact sets.  We believe that this can be deduced from the arguments in the proof of \cite[Theorem 4.1]{DMS}, but we will give an alternate proof here.  Again, since this might be of independent interest, we note that similar arguments show more generally:

\begin{lemma}
 Let $G$ be a finitely generated group acting acylindrically on a hyperbolic space $Z$, and fix a word metric on $G$ and an orbit map $\pi:G\to Z$. Then for every $D_1,D_2\geq 1$ there exists a Morse gauge $N$ so that the following holds. If $\gamma$ is a $(D_1,D_1)$--quasi-geodesic in $G$ such that $\pi\circ\gamma$ is a $(D_2,D_2)$--quasi-geodesic in $Z$, then $\gamma$ is $N$--Morse.
\end{lemma}

To see that $\partial_D X$ is contained in $\partial_M^{N'} X$ for some $N'$, let $\gamma$ be a representative of a point in $\partial_D X$, and let $\alpha$ be an $(L,L)$-quasi-geodesic in $X$ with endpoints on $\gamma$.  We claim that for a sufficiently large choice of $t$, there exists $R$, depending only on $D,L,t$ such that any subsegment of $\alpha$ lying outside the $t$-neighborhood of $\gamma$ has length at most $R$.  It will then follow that $\alpha$ lies in the $(t + R)$-neighborhood of $\gamma$.  

To show this, we will apply  Lemma 10.4 of \cite{MathieuSisto}.  
Up to replacing $\alpha$ with piecewise geodesic path lying within Hausdorff distance $2L$ from it, and replacing $L$ by $4L$, we can assume in addition that $\alpha$ is $L$--Lipschitz, a requirement for applying this lemma.  (This follows from the ``taming" process described in \cite[Lemma III.H.11]{BridsonHaefliger}).  

Say $\alpha'$ is a segment of $\alpha$ whose endpoints $a_1, b_1$ are distance $t$ from $\gamma$ and whose interior lies outside the $t$-neighborhood of $\gamma$.  Let $a_2,b_2$ be points on $\gamma$ at distance $t$ from $a_1,b_1$ and let $\gamma'$ be the segment of $\gamma$ from $a_2$ to $b_2$.  
Let $\bar a = \pi(a_2)$ and $\bar b=\pi(b_2)$ denote the endpoints of $\pi \circ \gamma'$. 
The Lemma states that there exists a divergent function $\rho$ and a constant $C$ such that 
$$(*) \hskip .2in \max\{l(\alpha'), l(\gamma')\} \geq  (d_T(\bar a,\bar b) - 2t -C) \rho(t). $$

Choose $t$ sufficiently large so that $t > C$ and $\rho(t) > 10DL$.  Set $R = 10DLt$ and suppose $\l(\alpha') > R$.
 (Without loss of generality, we may assume that all constants appearing in this argument are $\geq 1$.)
Since $\gamma$ lies in $\partial_D X$, $\gamma'$ intersects any $X_v$ in a segment of diameter at most $D$, so the argument above shows that
$$l(\gamma') = d_X(a_2,b_2) \leq D (d_T(\bar a,\bar b))+ 1). $$
And since $\alpha'$ is an $(L,L)$-quasi-geodesic, 
\begin{align*}
(**) \hskip .2in l(\alpha') &\leq L d_X(a_1,b_1) + L \\
& \leq L d_X(a_2,b_2)  + 2Lt +L \\
&\leq  DL d_T(\bar a,\bar b)+ DL +2Lt +L \\
& \leq DL(d_T(\bar a,\bar b) + 2t+2))
\end{align*}
We thus have
$10DLt < l(\alpha') \leq DL(d_T(\bar a,\bar b) + 2t+2))$,
from which we deduce that $d_T(\bar a,\bar b) > 8t-2 > 6t$, and hence
\begin{enumerate}
\item $d_T(\bar a, \bar b) + 2t +2 < d_T(\bar a, \bar b) +3t < 2 d_T(\bar a, \bar b)$, and
\item $d_T(\bar a, \bar b) -2t-C > d_T(\bar a, \bar b) -3t >  d_T(\bar a, \bar b)/2$.
\end{enumerate}
Applying (1) to the inequality $(**)$ gives 
$$l(\alpha')  < 2DL\, d_T(\bar a,\bar b) $$
Applying (2) to the inequality $(*)$ gives
$$\max\{l(\alpha'), l(\gamma')\} > 5DL \, d_T(\bar a,\bar b). $$
This is a contradiction, so we conclude that no such segment $\alpha'$ exists. 
This completes the proof of the theorem.
\end{proof}

\begin{varthm}[Corollary \ref{thm:graph_mflds}]
Let $M$ be a non-geometric graph manifold. Then $\partial_M \left(\pi_1(M)\right)$ is an $\omega$-Cantor space.
\end{varthm}

\begin{proof}
 Let $\mathcal{G}$ be the graph of groups decomposition of $\pi_1(M)$ associated to the geometric decomposition of $M$ along embedded tori and Klein bottles into finitely many Seifert manifolds. The vertices of the Bass--Serre tree associated to $\mathcal{G}$ correspond to the Seifert manifolds, which have empty Morse boundary and are undistorted in $\pi_1(M)$. Since $\pi_1(M)$ acts acylindrically on the Bass--Serre tree \cite[Theorem 7.27]{Abbott-Balasubramanya-Osin}, we can apply the previous theorem. 
\end{proof}

\section{Limits of \sier curves}

We now begin our study of limits of \sier curves. First, we state the characterization of the \sier curve that will be most useful for our purposes. We then define entwined \sier curves, and  study direct  limits of entwined \sier curves.

A \sier curve $S$ is a topological space homeomorphic to $\Sph-\bigcup_{i \in \N} {D_i}$, where $D_i$ is the interior of a closed disk in $\Sph$ satisfying
\begin{itemize}
\item $\overline{D}_i \cap \overline{D}_j = \emptyset$  all $i,j$.
\item $diam(D_i) \to 0$ as $i \to \infty$
\item $\bigcup \overline{D}_i$ is dense in $\Sph$. 
\end{itemize}
All \sier curves are homeomorphic to each other \cite{Whyburn}. The peripheral circles of a \sier curve are the boundaries of the $D_i$ as above, and they can be characterised as the subspaces $\gamma$ of $S$ homeomorphic to $S^1$ so that $S-\gamma$ is connected (this follows from the Jordan curve theorem, see the argument in \cite[page 256, case (ii)]{K:homeos_Sierpinski}).

\begin{definition}
 Let $S\subseteq T$ be \sier curves. We say that $S$ is \emph{entwined} in $T$ if no peripheral circle of $S$ intersects any peripheral circle of $T$.
\end{definition}

The following lemma says the following. Suppose that $S\subseteq T$ are \sier curves, with $S$ entwined in $T$. Then $T$ is obtained from $S$ by attaching \sier curves onto each peripheral circle.

\begin{lemma}\label{lem:Sierpinski_complement}
 Let $S\subseteq T$ be \sier curves, with $S$ entwined in $T$. Then there exist \sier curves $S_i\subseteq T$ with the following properties:
 \begin{enumerate}
  \item $T=S\cup \bigcup S_i$,
  \item $S_i\cap S$ is a peripheral circle of $S$ and $S_i$, and any peripheral circle of $S$ arises in this way,
  \item $S_i\cap S_j=\emptyset$ if $i\neq j$,
  \item in any metric on $T$ compatible with its topology, we have $diam(S_i)\to 0$.
 \end{enumerate}
\end{lemma}

\begin{proof}
 We can assume $T=\Sph-\bigcup {D_i}$, where the $D_i$ are open disks as described above. 
  
 Let $\gamma$ be a peripheral circle of $S$. Then $\Sph-\gamma$ is a union of two open disks $D_\gamma, D'_\gamma$ whose boundary is $\gamma$, by the Jordan curve theorem. 
 One of these two disks, say $D_\gamma$, is disjoint from $S$, for otherwise $S-\gamma$ would be disconnected.   Since $S$ is entwined in $T$,  each $D_i$ is contained in either $D_\gamma$ or $D'_\gamma$.  Thus, $\overline{D_\gamma}\cap T$ is $\Sph$ minus a union of interiors of closed disks, namely, $D'_\gamma$ and the $D_i$ contained in $D_\gamma$.  Notice that the diameters of the disks $D_i$ go to $0$, and that their union is dense in $\Sph$, and hence $\overline{D_\gamma}\cap T$ is a \sier curve, which we will denote by $S_\gamma$. We claim that the \sier curves $\{S_\gamma\}$ have the required properties.
 
 1) We first observe that $\Sph-S=\bigcup D_\gamma$. In fact, the closure of any complementary components of $S$ is a closed disk by \cite[Theorem 9]{Moore} and the Jordan curve theorem, and the boundary of such disk is then a peripheral circle. Thus, $\Sph = S \cup \bigcup \overline{D}_\gamma$ and hence $T = S \cup  \bigcup (\overline{D}_\gamma \cap T) 
 =S \cup \bigcup S_\gamma$.  
 
 2) We have $S_\gamma\cap S=\gamma$ by construction.
 
 3) It suffices to show that for two distinct peripheral circles $\gamma, \alpha$, we have $D_\gamma\cap D_\alpha=\emptyset$. If not, since the boundary circles of $\gamma$ and $\alpha$ are disjoint, up to switching $\gamma$ and $\alpha$, we would have $\overline{D_\gamma}\subseteq D_\alpha$ (given complementary components of $\gamma,\alpha$, they are either disjoint or one of them, as well as its closure, is contained in the other). But then we would have $\gamma\subseteq D_\alpha$, in contradiction with the fact that $D_\alpha\cap S=\emptyset$.
  
 4) This statement does not depend on the choice metric, so for convenience we endow $\Sph$ with the metric $d$ as a subspace of Euclidean space, and in turn endow $S$ and $T$ with their metrics subspaces of $(\Sph,d)$. Since each $S_\gamma$ is contained in the closure of the corresponding $D_\gamma$, it suffices to show that, given any $\epsilon>0$ there are only finitely many $\gamma$ with $diam(D_\gamma)\geq \epsilon$.
 
 Fix $\epsilon>0$, and we can further assume $\epsilon<1$. Notice that there exist only finitely many $\gamma$ of diameter at least $\epsilon/2$ (for example because $S$ is homeomorphic to the standard \sier carpet, and this is true of the peripheral circles in that case), and from now on we only consider those with diameter at most $\epsilon/2$. The issue we have to deal with now is that $D_\gamma$ can have diameter much larger than that of its boundary $\gamma$ (in fact, it is always the case that one of the two components of the complement of $\gamma$ is has ``large'' diameter). However, what we know, thanks to the fact that we are working with the restriction of the Euclidean metric, is that the component $D_\gamma$ of the complement of $\gamma$ either has diameter at most $\epsilon$, or it contains a ball in $(\Sph,d)$ of radius $1$. In fact, $\gamma$ is contained in a closed Euclidean ball $B\subseteq \mathbb R^3$ of radius $\epsilon/2$, and $B\cap \Sph$ is connected. Hence, either $D_\gamma$ is contained in $B$, yielding the first case, or $D_\gamma$ contains $\Sph-B$, which is easily seen to contain a half-sphere (since $\epsilon$ is sufficiently small), yielding the second case.
 
 The second case can occur at most finitely many times since, as we argued above, the $D_\gamma$ are pairwise disjoint, and we can only fit finitely many disjoint balls of radius $1$ on $\Sph$. Hence, there are only finitely many $\gamma$ with $diam(D_\gamma)\geq \epsilon$, as required.
\end{proof}

The following point-set topology lemma will allow us to check continuity of certain maps, namely the ones that we will encounter when constructing homeomorphisms between limits of entwined \sier curves.

\begin{lemma}\label{lem:piecewise_cont}
 Let $f: X \to Y$ be a map between two metric spaces and let $C_0,C_1,\dots$ be closed subsets of $X$. Suppose that
 \begin{enumerate}
 \item $X=\bigcup C_i$
  \item $C_0\cap C_i\neq \emptyset$ for all $i$,
  \item $diam(C_i)\to 0$,
  \item $diam(f(C_i))\to 0$,
  \item $f|_{C_i}$ is continuous for every $i$.
 \end{enumerate}
Then $f$ is continuous.
\end{lemma}

\begin{proof}
 Let $(x_n)$ be a sequence in $X$ converging to $x$. It suffices to show that a subsequence of $(f(x_n))$ converges to $f(x)$.
 
 We consider two cases, at least one of which applies, by property 1.
 
 1) There is a subsequence $(x_{n_k})$ contained in some $C_j$. Since $C_j$ is closed, we also have $x\in C_j$. Then $(f(x_{n_k}))$ converges to $f(x)$ because $f|_{C_j}$ is continuous by property 5.
 
 2) There is a subsequence $(x_{n_k})$ is so that no two elements of the subsequence are contained in the same $C_j$. Suppose $x_{n_k}\in C_k$. Let $x'_{n_k}\in C_k\cap C_0$ (which exists by property 2). Then $(x'_{n_k})$ converges to $x$ in view of property 3, which implies that $d_X(x'_{n_k},x_{n_k})$ tends to $0$. By case 1 (with $j=0$), we have that $(f(x'_{n_k}))$ converges to $f(x)$. Finally, by property 4, $d_Y(f(x_{n_k}),f(x'_{n_k}))$ also tends to $0$, so $(f(x_{n_k}))$ also converges to $f(x)$. 
\end{proof}

The following proposition is, essentially, the inductive step in the definition of homeomorphisms between limits of entwined \sier curves.

\begin{proposition}\label{prop:extend_homeo}
 Let $S\subseteq T, S'\subseteq T'$ be \sier curves, with $S$ entwined in $T$ and $S'$ entwined in $T'$. Then any homeomorphism $\phi:S\to S'$ can be extended to a homeomorphism $\psi:T\to T'$.
\end{proposition}

\begin{proof}
 Let $S_i\subseteq T,S'_i\subseteq T'$ be the \sier curves that we obtain from Lemma \ref{lem:Sierpinski_complement}.  Since $\{S_i \cap S\}$ is the set of peripheral circles of $S$ and 
 $\{S'_i \cap S'\}$ the set of peripheral circles of $S'$, we can choose the indices in such a way that 
 $\phi(S_i\cap S)=S'_i\cap S$.
 
 The arguments from \cite[pages 322-323]{Whyburn} show that any homeomorphism between peripheral circles of \sier curves can be extended to a homeomorphism between the \sier curves. In particular, there are homeomorphisms $\phi_i:S_i\to S'_i$ that extend $\phi|_{S_i\cap S}$. Hence, we can define $\psi:T\to T'$ by requiring $\psi|_S=\phi, \psi|_{S_i}=\phi_i$.   This is a well-defined map in view of Lemma \ref{lem:Sierpinski_complement}-(3). Then $\psi$ is clearly bijective.  By Lemma \ref{lem:Sierpinski_complement}-(1)-(2)-(4)), the hypotheses of Lemma \ref{lem:piecewise_cont} hold, so $\psi$ is continuous. Hence,  it is a homeomorphism because $T,T'$ are compact and Hausdorff. Thus, $\psi$ is the desired extension of $\phi$.
\end{proof}

Finally, we are ready for the main theorem of this section.

\begin{definition}
 An $\omega$-\sier curve is a topological space $\varinjlim_{i\in \mathbb N} X_i$, where each $X_i$ is a \sier curve and each $X_i$ is entwined in $X_{i+1}$.
\end{definition}

\begin{theorem}\label{thm:limit_Sierp}
 Any two $\omega$-\sier curves are homeomorphic.
\end{theorem}

\begin{proof}
 The proof is identical to the proof of Theorem \ref{thm:limit_Cantor}, using Proposition \ref{prop:extend_homeo} instead of Lemma \ref{lem:extend_homeo}.
\end{proof}

\section{$\omega$-\sier boundaries}

We now study Morse boundaries of fundamental groups of finite-volume hyperbolic $3$-manifolds. The main theorem of this section is:

\begin{varthm}[Theorem \ref{thm:Sierpinski_boundary}]
 Let $M$ be a finite-volume hyperbolic $3$-manifold with at least one cusp, and let $G=\pi_1(M)$. Then $\partial_M G$ is an $\omega$-\sier curve.
\end{varthm}

\subsection{Setup and notation} Let $M$ be a finite-volume hyperbolic $3$-manifold with at least one cusp, and let $G=\pi_1(M)$, regarded as a subgroup of $Isom(\mathbb H^3)$. Lifting disjoint cuspidal neighborhoods to $\mathbb H^3$, we obtain an equivariant family of disjoint horoballs $\{\mathcal H_p\}_{p\in P}$. Regarding $\Sph$ with the standard metric as the boundary of $\mathbb H^3$, we have the collection $P\subseteq \Sph$ of parabolic points for $G$. For each $p$, we denote $r_p=e^{-d_{\mathbb H^3}(x_0,H_p)}$, where $x_0$ is the origin of the Poincar\'e disk model of $\mathbb H^3$.

For a given $0<\lambda\leq 1$, we denote $V_\lambda=\Sph\setminus \bigcup_{p\in P} B(p,\lambda r_p)$. 
Set $X=\mathbb H^3-\bigcup H_p$, a neutered space for $M$. Denote by $O_p\subseteq X$ the horosphere that bounds $H_p$.
 
We will use the fact that $X$ is a $\mathrm{CAT}(0)$ space and that each $O_p$ is a flat in $X$, meaning that it is convex and isometric to $\mathbb R^2$, see e.g. \cite[Chapter II.11]{BridsonHaefliger}. In addition, the inclusion map $\iota:X\to\mathbb H^3$ is $1$-Lipschitz and proper.
 
 We will use this notation throughout this section.
 
\subsection{From the Morse boundary to $\partial \mathbb H^3$}

In this section we relate the Morse boundary of $X$ with $\partial \mathbb H^3$; the main result is Proposition \ref{prop:Morse=cheese}. Roughly speaking, we have to show that Morse geodesic rays in $X$ correspond to geodesic rays in $\mathbb H^3$ that do not spend too much time in horoballs, and also relate the latter rays with points in the subspace $V_\lambda \subset \partial \mathbb H^3$ defined above. This can be done in greater generality than our case (groups hyperbolic relative to subgroups with empty Morse boundaries, replacing $\mathbb H^3$ with the cusped space), but we decided to use the extra structure available to us to make the proofs simpler.

First of all, we show that Morse geodesic rays in $X$ give quasi-geodesics in $\mathbb H^3$.

\begin{lemma}\label{lem:retract_neutered}
Let $\gamma$ be an $N$-Morse geodesic ray in $X$. Then $\iota(\gamma)$ is a $(K,L)$-quasi-geodesic ray where $K, L$ depend only on $N$.
\end{lemma}

\begin{proof}
We will show that there exists a $(\lambda, \epsilon)$-coarse Lipschitz retraction of $\mathbb{H}^3$ onto $\iota(\gamma)$  where $\lambda, \epsilon$ depend on $N$. It follows that $\iota(\gamma)$ is a $(K, L)$-quasi-geodesic where $K,L$ depend on $\lambda, \epsilon$. 

Since $\gamma$ is $N$-Morse and $X$ is $\mathrm{CAT}(0)$, by \cite{charney-sultan}, we know that $\gamma$ is $D$-strongly contracting, where $D$ depends on $N$. 

Let $\pi:X\to \gamma$ be the closest-point projection in $X$. We claim that for each horosphere $O_p$, the diameter of $\pi(O_p)$ is bounded in terms of $N$. In fact, suppose that the said diameter is much larger than $D$ (recall that $D$ depends on $N$), so that there exist points $x,y\in O_p$ with far away projections to $\gamma$. The geodesic $\alpha$ from $x$ to $y$ is contained in $O_p$, and by strong contraction it has to pass within controlled distance of $\pi(x),\pi(y)$. Hence, $\gamma$ has a long subgeodesic, the one with endpoints $\pi(x),\pi(y)$, that lies in a controlled neighborhood of $O_p$. The length of this subgeodesic is, up to additive constants, the same as the diameter of $\pi(O_p)$. Since $O_p$ is a flat, and $\gamma$ is $N$-Morse, this length can then be controlled in terms of $N$, for otherwise we would find a quasi-geodesic in $O_p$ straying too far away from $\gamma$. 

Now, using the claim above, we can extend $\pi$ to $\pi':\mathbb H^3\to \gamma$ by mapping $H_p-O_p$ to any point in $\pi(O_p)$. It is readily seen that $\pi'$ is a coarsely Lipschitz retraction with controlled constants, as required.
\end{proof}

We now show that rays that do not spend much time in any horoball are exactly those with limits points in $V_\lambda$.

\begin{lemma}\label{lem:horoball_geometry}
 Fix a basepoint $x_0\in X$, let $\gamma$ be a geodesic ray in $\mathbb H^3$ starting at $x_0$ with limit point $x\in \partial\mathbb H^3$, and let $p\in P$.
 \begin{enumerate}
  \item For every $C$ there exists $\lambda$ with the following property. Suppose that $\gamma$ intersects the horoball $H_p$ in a set of diameter at most $C$. Then $x\notin B(p,\lambda r_p)$.\label{item:horob_to_cheese}
  \item For every $\lambda$ there exists $C$ with the following property. Suppose that $x\notin B(p,\lambda r_p)$. Then $\gamma$ intersects the horoball $H_p$ in a set of diameter at most $C$.\label{item:cheese_to_horob}
 \end{enumerate}
\end{lemma}

\begin{proof}

Recall that the standard metric on $\Sph=\partial \mathbb H^3$ can be written in terms of the Gromov product as $d_{\partial\mathbb H^3}([\alpha],[\beta])= e^{-(\alpha,\beta)_{x_0}}$.

(\ref{item:horob_to_cheese}) We have to show that, for a suitable $\lambda=\lambda(C)$, if $x\in B(p,\lambda r_p)$ then $diam(\gamma\cap H_p)\geq C$. If $x\in B(p,\lambda r_p)$, then
the Gromov product of $\gamma$ and the geodesic ray $\gamma_p$ is at least $d(x_0,H_p)-\ln(\lambda)$.  Thus when $\lambda$ is close to $0$, $\gamma$ and $\gamma_p$ stay within  $\delta$ of each other (where $\delta$ is the hyperbolicity constant of $\mathbb H^3$) for much longer than $d(x_0,H_p)$, so that $\gamma$ intersects $H_p$ in a set of large diameter. (Here we use that, up to bounded error, two rays $\delta$-fellow-travel for time equal to their Gromov product; for later purposes we note that the same holds replacing ``$\delta$'' with ``$2\delta$''.)
 
(\ref{item:cheese_to_horob}) We have to show that, given $\lambda$, there exists $C=C(\lambda)$ such that  if $diam(\gamma\cap H_p)\geq C$, then $x\in B(p,\lambda r_p)$.  Let $\gamma_p$ be as above, and let $\beta_p$ be the geodesic line with endpoints at infinity $x$ and $p$. Since horoballs are convex, $\gamma\cap H_p$ is a subgeodesic of $\gamma$, say with endpoints $a,b$. Since the Busemann function associated to $p$ is monotonic on $\gamma_p$ and $\beta_p$, we see that for $C$ large compared to $\delta$, up to switching $a$ and $b$, $a$ is $\delta$-close to $\gamma_p$ and $b$ is $\delta$-close to $\beta_p$ (since $a,b$ cannot be both close to, say, $\gamma_p$). 
 
Consider the triangle  with vertices $a,b,p$.  Since $a$ and $b$ lie on the horosphere bounding $H_p$, the midpoint $m$ of $[a,b]$ is equidistant from the other two sides (this is easily seen in the upper half space model of $\mathbb H^3$) and by the thin triangle condition, this distance is at most $\delta$.  It follows that the geodesic segment $[a,m]$ lies in the $\delta$-neighborhood of $[a,p]$ which, in turn, lies in the $\delta$-neighborhood of $\gamma_p$.  Thus $\gamma$ and $\gamma_p$ stay $2\delta$-close for a distance of at least $d(x_0,H_p)+ \frac{1}{2}C$.  From this we can see that for $C$ sufficiently large, the Gromov product $(\gamma, \gamma_p)_{x_0}$ will be at least $d(x_0,H_p) - \ln(\lambda)$, and hence the distance between $x$ and $p$ will be at most $\lambda r_p$. 
\end{proof}
 
The proof of the following lemma is a variation on the usual arguments to show that contracting properties of various kinds imply the Morse property. We will need it to show that certain rays in $X$, constructed from rays in $\mathbb H^3$, are Morse.

\begin{lemma}\label{lem:outside_neigh}
For every diverging function $f:\R_+\to\R_+$ there exists a Morse gauge $N$ with the following property.
 Let $Z$ be a geodesic metric space and let $\gamma\subseteq Z$ be a geodesic. Suppose that for each $R>0$ any path $\beta$ that intersects $N_R(\gamma)$ only at its endpoints $x,y$ has the property that $l(\beta)\geq f(R)d(x,y)-f(R)$. Then $\gamma$ is $N$-Morse.
\end{lemma}

\begin{proof} 
 Let $\alpha:[a,b]\to Z$ be a $(K,C)$-quasi-geodesic, so $l(\alpha|_{[t,u]})\leq K |t-u| +C$ for each $t,u\in[a,b]$. Choose $R$ so that $f(R)\geq K^2+K$. We will show that any subpath $\beta$ of $\alpha$ that intersects $N_R(\gamma)$ only at its endpoints $x=\alpha(t)$ and $y=\alpha(u)$ has length bounded by a constant $D$ depending only on $K,C,f(R)$, showing that $\alpha$ is contained in the $(R+ D)$-neighborhood of $\gamma$.
 
 We have,
 \begin{align*}
   K|t-u| +C \geq l(\beta) & \geq  f(R)d(x,y)-f(R) \\
  & \geq \frac{f(R)}{K}|t-u|-Cf(R)-f(R)\\
  & \geq (K+1)|t-u|-Cf(R)-f(R),
\end{align*}
 from which it follows that $|t-u|\leq Cf(R)+f(R)+C$. Setting $D=KCf(R)+Kf(R)+KC+C$, we conclude that  
$$ l(\beta)\leq KCf(R)+Kf(R)+KC+C=D$$  
as desired.
\end{proof}

Finally, we compare the spaces $\partial_M^N X$ with the spaces $V_\lambda$ and the describe how, in order to study $\partial_M X$, one can study the $V_\lambda$ instead. Hung Cong Tran has related results for general relatively hyperbolic groups \cite{Tran:Morse_rel_hyp}. However, for us the control on the size of the removed balls will be crucial, and this is not addressed in Tran's paper.

\begin{proposition}\label{prop:Morse=cheese}
 There exists a continuous injective map $\Psi: \partial_M X\to \Sph$ with the following properties.
 \begin{enumerate}
  \item For each Morse gauge $N$ there exist $\lambda =\lambda (N)$ so that $\Psi\left(\partial_M^{N} X\right)\subseteq V_\lambda$.\label{item:Morse_to_cheese}
  \item For every $0<\lambda\leq 1$, there exists a Morse gauge $N_\lambda$ so that for each $N\geq N_\lambda $ we have $V_\lambda\subseteq \Psi\left(\partial_M^{N} X\right)$.\label{item:cheese_to_Morse}
 \end{enumerate}
\end{proposition}

\begin{proof}
Throughout the proof, we fix basepoints $x_0=\iota(x_0)$ of $X$ and $\mathbb H^3$, and when discussing (quasi-)geodesic rays, we will assume that they are based at $x_0$.

We first define $\Psi \colon \partial_M X\to \Sph$. By Lemma \ref{lem:retract_neutered}, given $\ell \in \partial_M^{N} X$ and an $N$-Morse geodesic ray $\gamma$ representing $\ell$, we can define $\Psi(\ell)$ as the limit point of $\iota(\gamma)$. The fact that $\Psi$ is well-defined and continuous follows from the fact that if two $N$-Morse geodesic rays $\gamma,\gamma'$ have initial subgeodesics of length $L$ that stay within distance $C$ of each other, then $\iota(\gamma),\iota(\gamma')$ have the same property (since $\iota$ is $1$-Lipschitz). Injectivity follows from the fact that if the distance between $\gamma(t),\gamma'(t)$ diverges as $t$ goes to infinity, then the same is true for $\iota(\gamma),\iota(\gamma')$ since $\iota$ is a proper map.

In view of Lemma \ref{lem:horoball_geometry}, in order to prove item (\ref{item:Morse_to_cheese}), we have to show that given an $N$-Morse geodesic ray $\gamma$ in $X$, the geodesic ray $\alpha$ in $\mathbb H^3$ asymptotic to $\iota(\gamma)$ intersects each horoball $H_p$ in a set of diameter bounded by a constant depending only $N$. Since the Hausdorff distance between $\alpha$ and $\iota(\gamma)$ is bounded in terms of $N$, it suffices to bound the diameter of the intersection of $\iota(\gamma)$ with a suitable neighborhood of $H_p$.  In turn, since $\iota$ is a proper map, it suffices to do the same in $X$.  That is, we need to show that the diameter of the intersection of $\gamma$ with the $R$-neighborhood of $O_p$ is bounded in terms of $R$ and $N$.   To  see this, note that $O_p$ is a flat, and as the diameter of this intersection increases, we can find quasi-geodesic in the flat straying farther and farther away from $\gamma$.  Since $\gamma$ is $N$-Morse, the intersection must have bounded diameter.

In view of Lemma \ref{lem:horoball_geometry}, in order to prove item (\ref{item:cheese_to_Morse}), given a geodesic ray $\alpha$ in $\mathbb H^3$ that intersects any horoball $H_p$ in a set of diameter at most $C$, we have to find an $N$-Morse geodesic ray $\gamma$ in $X$ so that $\iota(\gamma)$ lies within finite Hausdorff distance of $\alpha$, where $N=N(C)$.

We first do the case $C=0$, and then show how to reduce to this case. If $C=0$, then $\alpha$ is contained in $X$, and moreover it is a geodesic ray in $X$ since $\iota$ is $1$-Lipschitz. We only have to argue that $\alpha$ is Morse in $X$, with controlled Morse gauge. By Lemma \ref{lem:outside_neigh}, it suffices to prove that whenever $\beta$ is a path (in $X$) intersecting the $R$-neighborhood of $\alpha$ only at its endpoints $x,y$, then $l(\beta)\geq f(R) d_X(x,y)-f(R)$, where $f$ is some fixed diverging function.  Consider such a path $\beta$, and regard it now as a path in $\mathbb H^3$. Since $\iota$ is proper, $\beta$ lies outside the $\rho(R)$-neighborhood of $\alpha$ (in $\mathbb H^3$), where $\rho$ is some diverging function. Moreover, $\iota$ is $1$-Lipschitz and hence the endpoints of $\beta$ are $R$-close to $\alpha$. Using the hyperbolicity of $\mathbb H^3$ (or even just that $\alpha$ is strongly contracting), one can then show that the length of $\beta$ is at least $g(\rho(R))d_{\mathbb H^3}(x,y)-2R$, for some diverging function $g$. Notice that $|d_{\mathbb H^3}(x,y)-d_X(x,y)|\leq 10R$, since $x,y$ lie $R$-close to $\alpha$ both in $X$ and in $\mathbb H^3$, and $\alpha$ is a geodesic in both $X$ and $\mathbb H^3$. Hence, we get the required inequality.

We are only left to reduce the case of a general $C$ to the case $C=0$, and we will do so by changing the neutered space. Given a geodesic ray $\alpha$ in $\mathbb H^3$ that intersects any horoball $H_p$ in a set of diameter at most $C$, we can regard it as a geodesic ray in a neutered space $X'$ containing $X$, in which we shrunk all the horoballs a uniform amount. The previous argument yields that $\alpha$ is a Morse geodesic ray in $X'$. There is a quasi-isometry $\phi:X'\to X$, whose constants depend only on $C$, that moves each point a bounded amount. Hence, $\phi(\alpha)$ is an $N$-Morse $(K,C)-$quasi-geodesic, with $N,K,C$ depending only on $C$. It is readily seen that a geodesic ray $\gamma$ within bounded Hausdorff distance from $\phi(\alpha)$ has the property that $\iota(\gamma)$ lies within finite Hausdorff distance of $\alpha$, and we are done.
\end{proof}

\subsection{Approximating strata with \sier curves}

As in the case of totally disconnected boundaries, we have to wiggle the strata in order for them to be actual \sier curves. The following proposition, which we prove later, is what will allow us to do this:
Recall that $P \subset \partial \mathbb H^3$ denotes the set of parabolic points for the group $G= \pi_1(M)$.  

\begin{proposition}\label{prop:circles}
 For each sufficiently small $0<\lambda\leq 1$, and each $p\in P$, there exists a circle $\gamma$ so that
 \begin{enumerate}
  \item $\gamma\subseteq B(p,3\lambda r_p/4)-B(p,\lambda r_p/4)=A$ and $\gamma$ is homotopically non-trivial in $A$,
  \item for each $p'\in P-\{p\}$ with $r_{p'}\leq r_p$ we have $\gamma\cap B(p',3\lambda r_{p'}/4)=\emptyset$.
  \end{enumerate}
\end{proposition}

\begin{cor}\label{cor:Sierpinski_between_cheese}
 For each sufficiently small $0<\lambda\leq 1$ there exists a \sier curve $S$ with $V_\lambda \subseteq S\subseteq V_{\lambda/4}$, with each peripheral circle of $S$ contained in some $B(p,3\lambda r_p/4)$.
\end{cor}

\begin{proof}
 Consider for each $p$ the open disk containing $p$ bounded by the circle $\gamma$ as in the previous lemma. Such disks are either disjoint or nested, so that we see that the complement of all the disks is the required \sier curve.
\end{proof}

We can now prove Theorem \ref{thm:Sierpinski_boundary}.

\begin{varthm}[Theorem \ref{thm:Sierpinski_boundary}.]
 Let $M$ be a finite-volume hyperbolic $3$-manifold with at least one cusp, and let $G=\pi_1(M)$. Then $\partial_M G$ is an $\omega$-\sier curve.
\end{varthm}

\begin{proof}
Since $G$ is quasi-isometric to the neutered space $X$, we can work with $\partial_M X$ instead. First of all, $\partial_M X$ is $\sigma$-compact since $X$ is $\mathrm{CAT}(0)$ by \cite[Main Theorem]{charney-sultan}. Hence, by Lemma \ref{lem:increasing_gauge} we have that $\partial_M X =\varinjlim \partial^{N}_M X$ can be chosen to be a countable limit over gauges $N_1, N_2, \ldots$. By Proposition \ref{prop:Morse=cheese} and Corollary \ref{cor:Sierpinski_between_cheese}, and provided that $N_1$ is large enough, there is a \sier curve $S_1$ so that $\partial^{N_1}_M X \subset S_1  \subset \partial^{N_{j(2)}}_M X$ for a sufficiently large $j(2)$. In fact, we can further require that $\lambda (N_{j(2)})<\lambda(N_1)/4$, for $\lambda(N)$ as in Proposition \ref{prop:Morse=cheese}. Applying Proposition \ref{prop:Morse=cheese} and Corollary \ref{cor:Sierpinski_between_cheese} again, we find $j(3)$ and another \sier curve $S_2$ so that $\partial^{N_{j(2)}}_M X \subset S_2  \subset \partial^{N_{j(3)}}_M X$, again with $\lambda (N_{j(3)})< \lambda(N_{j(2)})/4$. The condition $\lambda (N_{j(2)})<\lambda(N_1)/4$ ensures that $S_1$ is entwined in $S_2$. Proceeding inductively, we see that $\partial_M X$ is a limit of entwined \sier curves.  Thus, $\partial_M X$ is an $\omega$-\sier curve, as required.
\end{proof}

\subsection{Proof of Proposition \ref{prop:circles}}

We are only left to prove Proposition \ref{prop:circles}. We will follow arguments from \cite{MS:quasi_planes}, and the construction is roughly as follows. We start with a circle that possibly does not avoid all required balls around parabolic points, and we make detours to avoid balls of a certain size. We iterate the procedure for balls of smaller and smaller size, and then we take a limit. For technical reasons, we will work mostly with arcs rather than circles.

As in \cite{MS:quasi_planes}, in order to make this work we need machinery from \cite{Mac-08-quasi-arc}; we now state all relevant facts and definitions. First of all, after detouring, we would like the new circle to be close to the previous one, and this is captured by the following definition

\begin{definition}\label{def-iota-follows}
	For any $x$ and $y$ in an embedded arc $A$, let $A[x,y]$ be the closed, possibly trivial,
	subarc of $A$ that lies between them.
	
	An arc $B$ {\em $\iota$-follows} an arc $A$ if there
	exists a (not necessarily continuous) map $p:B \rightarrow A$, sending endpoints to endpoints,
	such that for all $x,\,y \in B$, $B[x,y]$ is in the $\iota$-neighborhood of
	$A[p(x),p(y)]$; in particular, $p$ displaces points at most $\iota$.
\end{definition}

In  \cite{Mac-08-quasi-arc}, Mackay defines a space $X$ to be \emph{$N$-doubling} if every ball can
be covered by at most $N$ balls of half the radius and \emph{$L$-linearly connected}  if for all $x,y \in X$ there exists a compact, connected subset $Y$ containing $x,y$ of diameter less than or equal to $Ld(x, y).$   

The following proposition will be used to remove unwanted detours from our circles (which in the limit might create unwanted topology). 

\begin{proposition}[{\cite[Proposition 2.1]{Mac-08-quasi-arc}}]\label{prop-coarsestr}
 Given a complete metric space $X$ that is
 $L$-linearly connected and $N$-doubling, there exist constants
 $s=s(L,N)>0$ and $S=S(L,N)>0$ with the following property:
 for each $\iota > 0$ and each arc $A \subset X$, there exists an arc
 $J$ that $\iota$-follows $A$, has the same endpoints as $A$,
 and satisfies
 \begin{equation} \label{eq-cqa}
  \forall x,y \in J,\  d(x,y) < s\iota \implies
   \diam(J[x,y]) < S\iota.
 \end{equation}
\end{proposition}

The following lemma will allow us to take limits.

\begin{lemma}[{\cite[Lemma 2.2]{Mac-08-quasi-arc}}]\label{lem-approx}
 Suppose $(X,d)$ is an $L$-linearly connected, $N$-doubling, complete
 metric space, and let  $s,\, S,\, \varepsilon$ and $\delta$
 be fixed positive constants satisfying
 $\delta \leq \min\{\frac{s}{4+2S},\frac{1}{10}\}$.
 If we have a sequence of arcs $J_1, J_2, \ldots, J_n, \ldots$ in $X$,
  such that for every $n \geq 1$
 \begin{itemize}
  \item $J_{n+1}$ $\varepsilon \delta^n$-follows $J_n$, and
  \item $J_{n+1}$ satisfies \eqref{eq-cqa} with
   $\iota = \varepsilon \delta^n$ and $s,\,S$ as fixed above,
 \end{itemize}
 then the Hausdorff limit $J = \lim_\mathcal{H} J_n$
 exists, and is an arc.
 Moreover, the endpoints of $J_n$ converge to the endpoints of $J$.
 \end{lemma}

Finally, we need a lemma about the geometry of horoballs.

\begin{lemma}\label{lem:separated_parabolics}
 There exists $\lambda_0$ so that for each $\lambda_1,\lambda_2\leq \lambda_0$, if $B(p,\lambda_1r_p)\cap B(p',\lambda_2r_{p'}) \neq \emptyset$ then either $r_{p'}\leq r_{p}/100$ or $r_p\leq r_{p'}/100$.
\end{lemma}

\begin{proof}
 In view of Lemma \ref{lem:horoball_geometry}, a point in $B(p,\lambda_1r_p)\cap B(p',\lambda_2r_p')$ is represented by a ray from the basepoint that spends a long time in both horoballs $H_p$ and $H_{p'}$.  If the geodesic ray goes through $H_p$ first (the other case is symmetric), then $H_{p'}$ is much further away from the basepoint than $H_p$.  Choosing an appropriate value for $\lambda_0$, we can make the segment inside $H_p$ as long as we want, and thus control the distance to $H_{p'}$. In particular, we can guarantee that $r_{p'}\leq r_{p}/100$.
\end{proof}

We are now ready to prove Proposition \ref{prop:circles}.

\begin{proof}[Proof of Proposition \ref{prop:circles}.]
 We can assume $\lambda\leq \lambda_0$, for $\lambda_0$ as in Lemma \ref{lem:separated_parabolics}.
 
 Start with the circle $\gamma_1$ of radius $\lambda/2$ around $p$. We think of $\gamma_1$ as the union of two arcs $J_1,J'_1$ joining diametrically opposite points.
 
 We will construct a sequence of arcs $J_n,J'_n$ with the same endpoints that have the following properties.
 
 \begin{itemize}
 \item $J_{n+1}$ satisfies \eqref{eq-cqa} of Proposition \ref{prop-coarsestr}, with $\iota=50^{-n}\lambda r_p$.
 \item $J_{n+1}$  $(5\times 50^{-n}\lambda r_p)$-follows $J_n$. 
  \item For each $p'\in P$ with $50^{-n-1} r_p < r_{p'}\leq r_p 50^{-n}$, we have $J_{n+1}\cap B(p',\lambda 50^{-n}r_{p})=\emptyset$.
 \end{itemize}
  and similarly for the sequence $J_n'$.

Assume we have constructed $J_i,J_i'$ for $i \leq n$.  We now construct $J_{n+1}$ and $J'_{n+1}$. If $J_n$ and $J'_n$ do not intersect any $B_{p'}=B(p',2\times 50^{-n}\lambda r_{p})$ with $50^{-n-1} r_p \leq r_{p'}\leq r_p 50^{-n}$, we can take $J_{n+1}=J_n$ and  $J'_{n+1}=J'_n$ and we are done (notice the ``2'' for later purposes). Otherwise, the $B_{p'}$ they intersect are disjoint by Lemma \ref{lem:separated_parabolics}. We can then replace each maximal segment of $J_n$ or $J'_n$ contained in some $B_{p'}$ with an arc along the boundary of $B_{p'}$. If the segment is initial or terminal, we ensure that we move the corresponding endpoints of $J_n$ and $J'_n$ to the same point. In this way we obtain immersed paths, which are not necessarily arcs, but we can restrict to arcs contained in the image of such paths. By applying this procedure, we obtain arcs $K_n$, $K'_n$ that $(4\times 50^{-n}\lambda r_p)$-follow $J_n$ and $J'_n$ (the constant is the upper bound on the diameters of the $B_{p'}$).
 
We now apply Proposition \ref{prop-coarsestr} with $\iota= 50^{-n}\lambda r_p$, obtaining arcs $J_{n+1}$ and $J'_{n+1}$, satisfying $(1)$, that $(50^{-n}\lambda r_p)$-follow $K_n$ and $K'_n$, and hence $(5\times 50^{-n}\lambda r_p)$-follow $J_n$ and $J'_n$. It is straightforward to check that $J_{n+1},J'_{n+1}$ have the required properties.

 According to Lemma \ref{lem-approx}, there are limit arcs $J$ and $J'$, that clearly share the same endpoints. While these may intersect in multiple points, there exist a pair of subarcs sharing only endpoints, that form a circle $\gamma$ around $p$. We now check that $\gamma$ satisfies the required properties.

 To check that $ \gamma\subseteq B(p,3\lambda r_p/4)-B(p,\lambda r_p/4)$, it suffices to prove the analogous containment for $J\cup J'$. Because of the second property of $J_{n+1}$, we have that $J_{n+1}$ is contained in the closed neighborhood around $J_1$ of radius
 $$5\lambda r_p (50^{-1}+50^{-2}+\dots )\leq \lambda r_p/9.$$
 A similar statement holds for $J'$, and since $J_1\cup J'_1$ is the circle of radius $\lambda r_p/2$ around $p$, we are done.
 
Moreover, the same computation as above also yields that $J$ $(\lambda r_p/9)$-follows $J_1$, and similarly for $J'$. One can use this to show that the concatenation $\gamma'$ of $J$ and $J'$ is homotopic to $\gamma_1$ in $A=B(p,3\lambda r_p/4)-B(p,\lambda r_p/4)$: Subdivide each into small arcs and retract arcs in $\gamma'$ to corresponding arcs in $\gamma_1$ along geodesics in $\partial \mathbb H^3$. Also, $\gamma'$ and $\gamma$ are homotopic since they only differ in the $2\lambda r_p/9$-neighborhood of the endpoints of $J,J'$.
 
 Hence, $\gamma$ is homotopically non-trivial.

 Now let $p'\in P-\{p\}$ with $r_{p'}\leq r_p$. If $r_{p'}\geq r_p/50$, then applying Lemma \ref{lem:separated_parabolics} with $\lambda_1=3\lambda/4$ and $\lambda_2=\lambda$, shows that $\gamma$ is disjoint from $B(p',\lambda r_{p'})$. Hence, suppose $50^{-n-1} r_p < r_{p'}\leq 50^{-n} r_p$, for some $n\geq 1$. Then we have that $J_{n+1}\cup J'_{n+1}$ does not intersect $B(p',\lambda 50^{-n}r_{p})$. Arguing as above, $J\cup J'$ is contained in the neighborhood around $J_{n+1}\cup J'_{n+1}$ of radius
 $$5\lambda r_p (50^{-n-1}+50^{-n-2}+\dots )\leq 50^{-n}\lambda r_p/10.$$  
 
Hence, $J\cup J'$ avoids the ball of radius $(50^{-n}\lambda r_{p})(1-1/10)\geq 3\lambda r_{p'}/4$ around $p'$, as required.
\end{proof}

\bibliographystyle{alpha}
\bibliography{omega_boundaries}
\end{document}